\documentclass{article}
\usepackage[T2A]{fontenc}
\usepackage[cp1251]{inputenc} 
\usepackage[english]{babel}
\usepackage{amsfonts,amssymb,mathrsfs,amscd,amsmath}
\usepackage{amsthm,graphicx}

\newtheorem{theorem}{Theorem}

\theoremstyle{definition}
\newtheorem{example}{Example}

\def\R{{\mathbb R}}
\def\N{{\mathbb N}}

\begin{document}

\title{On an invariant of pure braids}

\author{V.\,O.~Manturov, I.\,M.~Nikonov}
\date{}

\maketitle

Consider a pure braid $\beta$. It can be expressed as $n$ disjoint strands $X(t)=(x_1(t),\dots, x_n(t))\in(\R^2)^n$, $t\in[0,1]$ where $X(0)=X(1)$. We assume $\beta$ to be generic. This means that there are finite number of values $0<t_1<\cdots<t_m<1$ such that for $t\ne t_k$, $k=1,\dots, m$, neither four points of $X(t)$ lie on one circle or line, and the set $X(t_k)$, $k=1,\dots, m$ contains exactly one quadruple of points which lie on one circle (or line). Then for any regular value $t\ne t_k$, the set $X(t)$ determines a unique Delaunay triangulation $G(t)$.

Fix an integer $r\ge 3$. Denote $q=e^{\frac{i\pi}r}\in\mathbb C$. For any regular value $t$, consider the set $F_r(t)=F_r(G(t))$ of admissible colourings of the edges of the Delaunay triangulation $G(t)$, i.e. maps $f: E(G(t))\to\N\cup\{0\}$ such that for any triangle of $G(t)$ with edges $a,b,c$ one has
\begin{enumerate}
\item $f(a)+f(b)+f(c)$ is even;
\item $f(a)+f(b)\ge f(c)$, $f(a)+f(c)\ge f(b)$, $f(b)+f(c)\ge f(a)$ ;
\item $f(a)+f(b)+f(c)\le 2r-4$.
\end{enumerate}
Let $V_r(t)$ be the linear space with the basis $F_r(t)$.

\begin{figure}[h]
\centering\includegraphics[width = 0.4\textwidth]{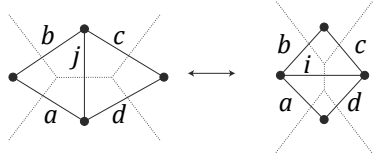}
\caption{Transformation of Delaunay triangulation}\label{tiling_change}
\end{figure}

Consider a singular value $t_k$. At the moment $t_k$ a transformation (flip) of the triangulation occurs (Fig.~\ref{tiling_change}). By recoupling theorem~\cite[Theorem 6.2]{KL} one can define a linear operator $A_r(t_k): V_r(X(t_k-0))\to V_r(X(t_k+0))$ whose coefficients are $q-6j$-symbols. More detailed, the triangulations $G(t_k-0)$ and $G(t_k+0)$ differ by an edge: $j=E(G(t_k-0))\setminus E(G(t_k+0))$, $i=E(G(t_k+0))\setminus E(G(t_k-0))$. Let $a,b,c,d$ be the edges incident to same triangles as $i$ or $j$. For any $f\in F_r(t_k-0)$ considered as a basis element of $V_r(t_k-0)$, we set
\begin{equation}\label{eq:action}
A_r(t_k)f=\sum_{g\in F_r(t_k-0)\colon f=g|_{E(G(t_k-0))\cap E(G(t_k+0))}}
\left\{\begin{array}{ccc}
g(a) & g(b) & g(i)\\
f(c) & f(d) & f(j)
\end{array}\right\}_q\cdot g,
\end{equation}
where $\left\{\begin{array}{ccc}
\alpha & \beta & \xi\\
\gamma & \delta & \eta
\end{array}\right\}_q$ is the $q-6j$-symbol (for the definition see~\cite[Proposition 11]{KL}).

Let
\begin{equation}\label{eq:operatorA}
A_r(\beta)=\prod_{k=1}^m A_r(t_k): V(X)\to V(X)
\end{equation}
 be the composition of operators $A(t_k)$.

\begin{theorem}\label{thm:recoupling_braid_invariant}
The operator $A_r(\beta)$ is a pure braid invariant.
\end{theorem}

\begin{proof}
Consider a pure braid isotopy $\beta_s$, $s\in[0,1]$. We can assume the isotopy is generic. This means that the isotopy includes a finite number of values $s_l$ such that the braid $\beta_{s_l}$ has a unique value $t^*$ such that the set $X_s(t^*)$ contains one of the following singular configurations
\begin{enumerate}
\item five points which belong to one circle (or line);
\item two quadruples of points which lie on one circle. The circles for the quadruples are different.
\end{enumerate}
The first case leads to the pentagon relation for Delaunay triangulation transformations (Fig.~\ref{flip_5-gon}). This relation induces a pentagon relation for the operator $A_r(\beta)$ which holds by theorem~\cite[Proposition 10]{KL}.
\begin{figure}[h]
\centering\includegraphics[width = 0.6\textwidth]{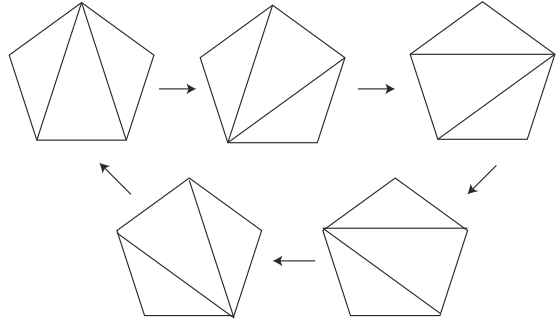}
\caption{Pentagon relation}\label{flip_5-gon}
\end{figure}

The second case leads to commutativity relation $A_1A_2=A_2A_1$ where $A_1$, $A_2$ correspond to flips on the given quadruples of points. The relation holds by construction of the operators $A_1$, $A_2$.
\end{proof}

Thus, we get a construction analogous to that in~\cite[Section 16.6]{FKMN} which uses Ptolemy relation.

\begin{example}
Consider the point configuration $P_1=(0,0)$, $P_2=(\frac 13,0)$, $P_3=(-\frac 12,\frac{\sqrt 3}2)$, $P_3=(-\frac 12,-\frac{\sqrt 3}2)$, $P_5=(1,0)$. Consider the pure braid $\beta$ given by dynamics of $P_2$ moving around $P_1$: $P_2(t)=(\frac 13\cos t, \frac 13\sin t)$, $t\in[0,2\pi]$. The dynamics produces six different Delaunay triangulations. For $r=4$, each triangulation has $160$ admissible colourings. The operator $A_4(\beta)$ of dimension $160$ is not identity (it has the eigenvalue $-1$ of multiplicity $20$). Thus, the invariant $A_r$ is not trivial.
\end{example}

Let us define a generalization of the construction described above. Consider a pentagon tuple $\Pi=(X,C,\{\cdot\})$ where
\begin{itemize}
\item $X$ is an arbitrary set (set of~\emph{labels});
\item $C\subset X^{\times 3}$ is a subset (\emph{admissible condition set}) which is invariant under the natural action of the symmetry group $\Sigma_3$ on $X^{\times 3}$;
\item $\{\cdot\}:X^{\times 6}\to\R$ is a function (\emph{$6j$-symbol}) such that
\begin{itemize}
\item $\left\{\begin{array}{ccc}
a & b & i\\
c & d & j
\end{array}\right\}\ne 0$ implies $(a,b,j),(c,d,j),(a,d,i),(b,c,i)\in C$;
\item
\[
\sum_{i\in X}  \left\{\begin{array}{ccc}
a & b & i\\
c & d & j
\end{array}\right\}
\left\{\begin{array}{ccc}
d & a & k\\
b & c & i
\end{array}\right\}=
\left\{
\begin{array}{cl}
1,& k=j\mbox{ and }(a,b,j),(c,d,j)\in C,\\
0,& \mbox{otherwise};
\end{array}
\right.
\]
\item
\begin{multline*}
\sum_{m\in X}  \left\{\begin{array}{ccc}
a & i & m\\
d & e & j
\end{array}\right\}
\left\{\begin{array}{ccc}
b & c & l\\
d & m & i
\end{array}\right\}
\left\{\begin{array}{ccc}
b & l & k\\
e & a & m
\end{array}\right\}=\\
\left\{\begin{array}{ccc}
b & c & k\\
j & a & i
\end{array}\right\}
\left\{\begin{array}{ccc}
k & c & l\\
d & e & j
\end{array}\right\}
\end{multline*}
\end{itemize}
\end{itemize}

Given a pentagon tuple $\Pi=(X,C,\{\cdot\})$, let us define a pure braid invariant. For a Delaunay triangulation $G(t)$, consider the set $F_\Pi(t)$ of admissible colourings, i.e. maps $f: E(G(t))\to X$ such that for any triangle with edges $a,b,c$ the condition $(f(a),f(b),f(c))\in C$ holds. Let $V_\Pi(t)$ be the linear space generated by $F_\Pi(t)$.

For a pure braid $\beta$, define the operator $A_\Pi(\beta)\in End(V_\Pi(0))$ by the formulas analogous to the formulas~\eqref{eq:action} and~\eqref{eq:operatorA}.

\begin{theorem}
The operator $A_\Pi(\beta)$ is a pure braid invariant.
\end{theorem}

The proof repeats the proof of Theorem~\ref{thm:recoupling_braid_invariant}.

\begin{example}[Ptolemy relation]
$\Pi=(\R,C_{Pt},\{\cdot\}_{Pt})$ where $C=\R^{\times 3}$ and
\[
\left\{\begin{array}{ccc}
a & b & i\\
c & d & j
\end{array}\right\}=
\left\{
\begin{array}{cl}
1,& ab+cd=ij,\\
0,& ab+cd\ne ij;
\end{array}
\right.
\]
\end{example}

\begin{example}[Recoupling theory]
Fix an integer $r\ge 3$. Consider the tuple $\Pi=(\N\cup\{0\},C_{rc},\{\cdot\}_{rc})$ where $C_{rc}$ are the triples $a,b,c$ which obey the conditions of admissible colorings described above,
$\left\{\begin{array}{ccc}
a & b & i\\
c & d & j
\end{array}\right\}
$
are the $q-6j$ symbols defined above.
\end{example}

We express our gratitude to Louis H.Kauffman, V.G.Turaev and E.A.Mudraya for fruitful and stimulation discussions and Nikita D.Shaposhnik for pointing out some remarks.

%
%
%
%
%

\end{document}